\documentclass[a4paper,12pt]{article}
\usepackage[top=30mm, bottom=30mm, left=30mm, right=30mm]{geometry}
\title{Continuous Branching Processes with Settlement in Cancer Metastasis: Stochastic Modelling and the Feller Property}

\author{Ivan Biočić, Bruno Toaldo, Lena Zuspann}
\date{}

\usepackage[utf8]{inputenc}
\usepackage{fullpage}

\usepackage{amsmath}
\usepackage{amsthm}
\usepackage{amsfonts,dsfont}
\usepackage{mathtools}
\usepackage[dvipsnames]{xcolor} 
\usepackage{amssymb} 
\usepackage{enumerate}
\usepackage[shortlabels]{enumitem}
\usepackage{comment}
\usepackage{aligned-overset}
\usepackage{yfonts}
\usepackage{cancel}
\usepackage{xcolor}
\usepackage{fancyvrb}

\numberwithin{equation}{section}

\usepackage{hyperref}
\hypersetup{
	colorlinks=true,
	linkcolor=magenta,
}

\newtheorem{theorem}{Theorem}[section]

\newtheorem{definition}[theorem]{Definition}
\newtheorem{lemma}[theorem]{Lemma}
\theoremstyle{definition}
\newtheorem{remark}[theorem]{Remark}
\newtheorem{example}[theorem]{Example}

\newtheorem*{assumption*}{\assumptionnumber}
\providecommand{\assumptionnumber}{}
\makeatletter
\newenvironment{assumption}[2]
{%
	\renewcommand{\assumptionnumber}{(\textbf{#1#2})}%
	\begin{assumption*}%
		\protected@edef\@currentlabel{(\textbf{#1#2})}%
	}
	{%
	\end{assumption*}
}
\makeatother

\DeclareMathOperator\supp{supp}	

\newcommand{\R}{{\mathbb R}}

\newcommand{\N}{{\mathbb N}}

\renewcommand{\P}{\mathds P}
\newcommand{\E}{{\mathds E}}

\newcommand{\EXP}{\textrm{Exp}}
\usepackage{dsfont}
\newcommand{\1}{\mathds{1}}

\renewcommand{\AA}{\mathcal{A}}
\newcommand{\BB}{\mathcal{B}}
\newcommand{\MM}{\mathcal{M}}

\newcommand{\KK}{\mathcal{K}}
\newcommand{\DD}{\mathcal{D}}

\newcommand{\LL}{\mathcal{L}}
\newcommand{\FF}{\mathcal{F}}

\newcommand{\XX}{\mathcal{X}}

\newcommand{\wt}{\widetilde}

\newcommand{\PAR}{\textrm{par}}
\newcommand{\Root}{\{\emptyset\}}

\allowdisplaybreaks

\begin{document}
	\maketitle

	\begin{abstract}
		Motivated by models of cancer metastasis, this paper introduces a type of (multi-type) branching process that records the positions of particles, representing tumor cells or clusters. Particles may be absorbed (removed from the state space), move, or settle. The process is rigorously constructed, and the Markov property is established via embedding into a multidimensional process that tracks the labels, positions, and phases (moving or resting) of living particles. The Feller property for the associated semigroup is investigated. It is proved for a simplified model that tracks the number of particles in each class, and an explicit generator is derived, enabling Feynman-Kac-type formulas in this framework.
	\end{abstract}
	


\section{Introduction}

Understanding how local growth mechanisms interact with rare long-range dispersal events to produce large-scale patterns is a central problem in many spatially structured systems \cite{Hastings2004, Lewis2016}. Examples include seed dispersal and plant colonisation in ecology \cite{Kim2022}, range expansion and evolution in animal populations \cite{Duchen2020, Wrensford2025, He2026} as well as the spread of forest fires through ignition cascades \cite{Zacharakis2023}. Although these processes differ in biological or physical detail, from a mathematical perspective, these dynamics are naturally represented through branching systems that record both spatial position and ancestry, which is genealogical in biological settings (offspring, mutations) and causal in physical ones (ignitions, infections). Thus, such models provide a unified framework for studying how local mechanisms and stochastic dispersal combine to shape global propagation.

Metastatic spread in cancer also exhibits a combination of such multi-scale temporal and spatial mechanisms \cite{Durrett2015, Avanzini2019}. Biologically, the process involves the dissemination of cells from a primary tumour, entering the vascular system, their survival in transport, and rare successful colonisation of distant organs; an ordered phenomenon known as the invasion-metastatic cascade \cite{Hesketh2013}. The complexity and interplay of different mechanisms cause metastases to account for about 80\% of cancer related deaths \cite{Gupta2006, Hanahan2011}. Thus, it is crucial for mathematical modelling in this area to accurately cover these underlying processes which are stochastic in nature, spatially structured and involve genealogical evolution \cite{Franssen2019}. Modelling efforts have made significant progress on individual aspects of the invasion-metastatic cascade \cite{BELLOMO2008, Scott2013}, such as primary tumour growth and invasion \cite{Enderling2014, ElHachem2021, Blanco2021, Sfakianakis2021, Katsaounis2023, Hu2024}, vascular transport \cite{Lenarda2019, Martinson2021, Ezeobidi2025}, or genetic evolution of advantageous mutations \cite{Reiter2017, Heyde2019, Hirsch2025}. Moreover, recent work has attempted to capture metastatic spread entirely in single network-type models with heavy-tailed transition probabilities between grid-like sites representing different organs \cite{Franssen2019, Yamamoto2019, Singh2025, Wieland2025}. However, such modelling efforts require introducing strong simplifications with respect to models covering only a particular step of the invasion-metastatic cascade in order to study analytical properties or appropriately simulate the model \cite{Franssen2019, Singh2025}. For this reason, we adopt a different perspective: we aim at drawing from well-established spatially and temporally local modelling approaches to inform global dynamics such as proliferation, spreading and settlement rates. These rates can then be incorporated into a global model capturing genetic lineages relating cancer cells to each other, while simultaneously tracking their spatial positions when leaving a primary tumour and surviving to form metastases. This framework then remains mathematically rigorous and analytically tractable, while incorporating more of the underlying biological mechanisms. 

In this work, we focus on rigorously setting up the global model for our proposed approach. Concretely, our model builds on the conceptual idea of \cite{Frei2019}, who introduced a particle-based continuous-time spatial branching process with settlement, where particles correspond to cancer cell clusters. We extend their approach to allow for additional properties of particles to be considered, which is highly relevant when incorporating more specific dynamics of metastatic spread, but also of other possible applications. Moreover, we situate our model within the framework of classic branching theory \cite{Harris-branching1963, AthNey, Li2011, Kyprianou2014} which allows us to rigorously investigate its probabilistic structure. Indeed, it can be seen that our model is related to classical multi-type (Markov) branching processes \cite{AthNey, Fittipaldi2022, Horton2023, Li2024}; however we will include a spatial component tracking the position of each particle. As a first step, we prove that our process retains the Markov property on an appropriate state space incorporating type and spatial position. Under a slight reformulation, the corresponding semigroup is then shown to be Feller. Furthermore, we derive an explicit formulation of its generator which opens the door to different areas of probabilistic-analytical investigations. In particular, one can then look at Feynman-Kac-type results connecting expectations over particle trajectories to solutions of linear parabolic equations \cite{DelMoral2000}. Extending to nonlinear settings, the associated Feller semigroup can be used to obtain McKean-type representations, which express solutions of nonlinear reaction-diffusion equations as expectations over the entire branching population \cite{Berestycki2014}. More generally, the generator and scaling limits naturally connect our branching model to the broader theory of superprocesses \cite{Etheridge2000, Chen2007}.

Of special interest in the context of cancer modelling is the analytical connection via the McKean-representation to the Fisher-Kolmogorov-Petrovskii-Piskunov (Fisher-KPP) equation, which is widely used to model front propagation phenomena, particularly in tumour invasion \cite{Gatenby1996, Gerlee2016, Colson2021, Lorenzi2024} among others \cite{Martin2016, Chen2020, Painter2023, Simpson2024}; our results consist in a first step forward in the direction of establishing this connection for the general processes introduced here. We remark that, also in this regard, our results provide an extension of the framework in \cite{Frei2019}, as they investigate a connection to mild solutions of integro-differential equations which are less common in the context of modelling cancer spread.

The remainder of this paper is organised as follows. We give a heuristic description of the proposed branching process model in the remainder of Section 1. Section 2 provides the rigorous construction of our modelling framework, where the respective elements are introduced in full generality and we provide examples on how they translate to metastatic spread. This deliberate choice makes our model readily applicable to other rare long-range dispersal phenomena. In Section 3, we provide an explicit upper bound on the expected amount of alive particles at a fixed time and prove the Markov property for the branching process. Section 4 then discusses the Feller property of the associated semigroup for the branching process and a simplified version of it. Finally, in Section 5, we derive an explicit formulation of the generator.

	\subsection{Heuristic description of the metastatic branching process}\label{ss:heuristic}
    In this model, we start with one initial tumour cell (or a cluster of cells) moving randomly through the body as a stochastic process $(Z_t)_t$. After an exponential random time with the rate $\mu_S$, it is settled at the current location. After the tumour cell/cluster is settled, it can shed (give birth to) new cells/clusters into the circulatory system of the body, where the shedding happens at a given rate $\mu_B$. Each newly shed cell/cluster starts its movement from the position of its settled parent, and moves independently of all other cells/clusters according to the same law of the stochastic process $(Z_t)_t$, then gets settled and sheds new cells/clusters.
	
	Here, we also consider the possibility that a given cell/cluster can die before it settles (according to an exponential random time with the given rate $\delta_M$, or after hitting some fixed set), and therefore it cannot shed new cells/clusters. Here, the death by hitting some set may have an interpretation that if the cell/cluster reaches some specific set of organs it is flushed out of the system. Also, the settled particle can die after some independent exponential random time with rate $\delta_S$ or when the settled particles sheds more then $L$ offsprings in total, and obviously in that case it also stops shedding new cells/clusters. Here, the death triggered by exceeding a total of $L$ offsprings can be interpreted as the settled particle having exhausted all resources necessary for further reproduction.

    Our construction of the model allows us to consider also the case when the initial tumour cell is settled at the beginning. Both are biologically reasonable since starting by a settled tumour cell correspond to the time of its detection, while a moving tumour cell corresponds to detecting a tumour e.g. through the concentration in a patient's blood draw.

    The metastatic branching model described above can be inserted in the theory of multi-type Markov branching processes, as in \cite[Chapter V, Section 7]{AthNey}, where the emphasis is made on the number of cells of each type. Here, on the other hand, it is of the greatest importance, with respect to the biological/oncological motivation, to study the positions of the cells, and not only the number of cells of each type.

    The rigorous construction and some additional interpretations of the described metastatic branching process is done in Section \ref{s:construction}. To see more vividly the described evolution, it is instructive to look at Figure \ref{fig:PossibleDynamics}.
	\begin{figure}[ht]
			\centering
			\includegraphics[width=\textwidth]{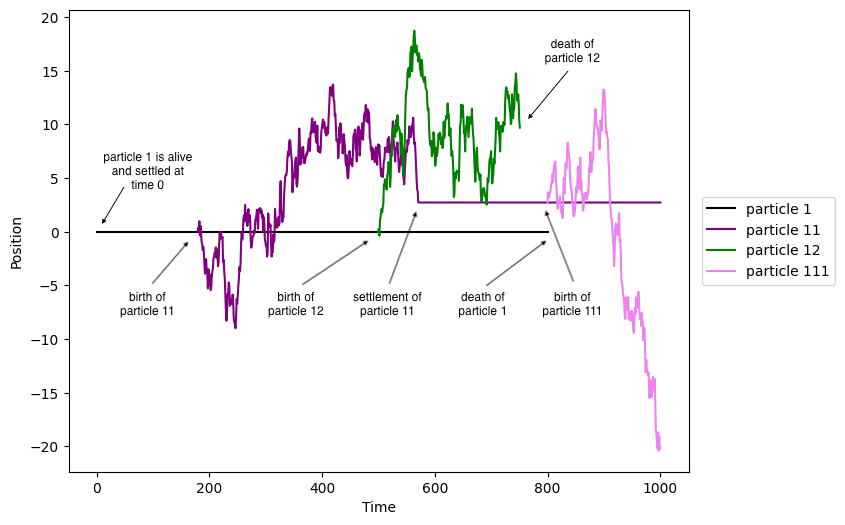}
			\caption{Possible dynamics of the metastatic branching process}
			\label{fig:PossibleDynamics}
		\end{figure}

	\section{Construction of the metastatic branching process}\label{s:construction}

	In this section, we rigorously construct the metastatic branching process heuristically described in the previous section. As we mentioned there, in cancer modeling, it is relevant to consider the positions of moving particles, not just the number of particles of each type. Therefore, we construct a Markov process here that includes the position of each alive particle in its coordinates. The branching process whose rigorous construction is provided in this section is a general object including previous models arising in applications (e.g., \cite{Frei2019}) as particular cases of our unifying framework.
	
	In the construction, we use nomenclature of the branching theory since we deal with quite a general model, with possible applications not only in biology. In particular, we will treat tumour cells/clusters simply as (abstract) particles. The comments and remarks on the introduced objects, however, will remain in the biological nomenclature. 
	
	\subsection*{Labeling}
	First we describe the labeling of the particles. We will use the Ulam-Harris-Neveau-notation, i.e. the labels are elements of
	\begin{equation}\label{eq:DefParticleLabelling}
		\mathcal{L} \coloneqq\bigcup_{n=0}^\infty\N^n.
	\end{equation}
	
	Using the convention $\N^0 = \Root$, an element of $\mathcal{L}$ consists of a finite sequence $l = (l_1, \dots, l_n)$ of integers for which we introduce the following widely accepted notions:
    the generation of $l$, denoted by $\vert l \vert$, is given by $n$;
	the history of $l$ up to the generation $m\in\N$, $m\le n$, is denoted by $l_{\vert m}\coloneqq (l_1, \dots, l_m)$;
    the parent label is given by the mapping $\PAR: \mathcal{L}\backslash \emptyset \rightarrow \mathcal{L}$ with $\PAR(l) = \PAR\left((l_1, \dots, l_n)\right) \coloneqq (l_1, \dots, l_{n-1})$;
    the continuation of a genealogical line: if two labels $k=(k_1, \dots, k_m)$ and $l=(l_1, \dots , l_n)$ both belong to $\mathcal{L}$, we write $kl = (k_1, \dots , k_n, l_1, \dots, l_m)$, so in particular, we have for any $l\in\mathcal{L}$ that $l\emptyset = \emptyset l = l$.

	With this notation, we define an ordering of these labels within a tree structure, cf. \cite[Definition 1]{Berestycki2014}.
	\begin{definition}\label{def:tree}
		A (locally finite, rooted) plane tree $T$ is a subset of $\mathcal{L}$ satisfying
		\begin{enumerate}[(a)]
			\item $\emptyset \in T,$
			\item $l \in T\backslash \left\{\emptyset\right\} \, \Rightarrow \, \PAR(l) \in T,$
			\item For every $l\in T$, there exists an integer $A_l(T) =:A_l \geq 0$ such that for every $i\in \N$, $li\in T$ if and only if $i \leq A_l$.
		\end{enumerate}
		Moreover, we denote the set of all (locally finite, rooted) plane trees by $\mathcal{T}$.
	\end{definition}
    The particle $\Root$ is usually called the root.

	\subsection*{Particle movement}
    We will consider the particle movement  in a locally compact, separable metric space $E$ (with the corresponding metric $d_E$). This choice allows us to formulate classical models (such as \cite{Frei2019}) in a more general setting while, in terms of applications, it gives more possibilities, see Remark \ref{r:space-movement}.
    
    All particles, independently of each other, follow the movement of the Markov process $Z=(Z(t))_t$ in $E$ with c\`adl\`ag paths. By $Z_p=(Z_p(t))_t$ we will denote the movement of the particle $p$ in $E$, while with $Z$ we will refer to the generic movement of this Markov process.
    
    \begin{remark}\label{r:space-movement}
        The movement of the particle $p$ denoted by $Z_p$ as above, does not have to describe its position in a strict sense. E.g. for the tumour cell, we can have $E=\R^3$ denoting its position. But if we also want to track the tumour cluster, it is important to track its size, so in this case we can have a Markov process in $E=\R^3\times [0,\infty)$ where the part in $\R^3$ serves as the position of the tumour cluster, and the part in $[0,\infty)$ is the cluster size. In other words, $E$ can serve as the space for a set of characteristics of a particle. The restriction here is that the change in time of such characteristics should satisfy the Markov property. 
        
        An example could be a particle's genetic information and tracking mutations thereof over time. In practice, this is highly relevant as some drugs rely on specific genetic mutations to be able to bind to the cancer cells. In terms of modelling, changes to this characteristic can be viewed as a discrete-state Markov chain with switching probabilities corresponding to the likelihood of a particular mutation occurring from another. 
    \end{remark}

    Here we are allowing that the particle $p$ dies on its own, in the sense that $Z$ does not have to be a conservative process. More precisely, we consider the process $Z$ which in fact lives on $E\cup \partial$, where $\partial$ is an added point to $E$, called the cemetery, where the process moves after its death, i.e. we allow $\P(Z(t) \in E|Z(0)=x)<1$.

    \subsection*{Particle settlement}
    In the heuristic description of the metastatic branching process, we said that each particle (i.e. a tumour cell/cluster) may settle in the current position  after independent exponential time with rate $\mu_S$. To be able to track such behavior, in view of the desired Markov property, we define the process $\theta_p=(\theta_p(t))_t$ such that $\theta_p(t)=0$ if the particle $p$ is not settled, and $\theta_p(t)=1$ if it is. To be coherent with the heuristic description of the metastatic branching process, we assume that $\theta_p$ is right-continuous and non-decreasing such that $\theta_p(t)\in\{0,1\}$ for all $t>0$. In other words, there is only one jump of $\theta_p$ and in that time, the particle $p$ stops moving and gets settled in its current position.

    \begin{remark}
        Settlement in the case of modelling metastatic spread of cancer corresponds to the moving cancer cells extravasating from the bloodstream to establish a new colony. Note that, in practice, there are multiple underlying biological processes involved determining whether a settlement attempt of a cancer cell cluster is successful in the sense that the cells survive long enough in the hostile environment of another organ (with e.g., less or different kinds of nutrients present) to proliferate and eventually form another tumour. 

        Modelling all of these contributing factors specifically would result in a model impossible to analyse analytically or simulate computationally. This motivates to instead consider a global scale model depending on a small number of interpretable parameters, in this case exponential waiting times, which could then be inferred from looking at local dynamics given these concrete biological mechanisms.
    \end{remark}
	
	\subsection*{Particle reproduction}
    Contrary to classic definitions of branching processes as introduced in \cite{Berestycki2014}, in our model the parent particle does not die at the birth of its children, but stays alive and can even continue to have more children. We model this in view of \cite{Kyprianou} by introducing for each particle $p\in\tau$ a reproduction process on the real line, denoted by $\xi_p$. Since reproduction in the applications can only happen after settlement and is governed by exponential waiting times with parameter $\mu_B \geq 0$, we will assume that for each particle $p$, independently of all others, follows the reproduction process $\xi_p=(\xi_p(t))_t$ which is a compound Poisson process started in a random position $\xi_p(0)=\xi_0$ in $\N_0$ (i.e. $\xi_0$ is a random variable with values in $\N_0$). 
    
    We use the notation
    \begin{align}\label{eq:xi0 distribution}
        \xi_0\sim
        \begin{pmatrix} 0 & 1 & 2 &\dots\\q_0 & q_1 &q_2 & \dots
    \end{pmatrix},
    \end{align}
    where it is allowed that $q_k=0$ for some or all $k\in \N_0$. The jumps sizes of the compound Poisson process $\xi_p$ are in $\N$, i.e. the L\'evy measure $\nu(ds)$ of $\xi_p$ is supported on $\N$ and $\nu(0,\infty)=\nu(\N)=\mu_B$. 
    
    Such reproduction mechanism says the following: at the settlement time, the particle immediately gives birth  to $\xi_0$ new particles (a random number, independent of the mechanism that follows), and after each of (independent) exponential waiting times with rate $\mu_B$, the settled particle $p$ gives a random number of new offsprings, where this number of offsprings follows the distribution $\nu(\cdot)/\mu_B$.
    
    We will use the following notation: for all $p\in \LL$ and $i\in\N_0$, let $\phi_{pi}$ denote the $i$-th birth time of particle $p$, where the time starts at the time of the settlement of the particle $p$. In particular, we have $0\leq \phi_{p1} \leq \phi_{p2} \leq \dots $, where $ \phi_{p1}=0$ holds in the case if $\xi_0>0$.

    \begin{example}
        A simple model where the reproduction is governed by exponential waiting times (with parameter $\mu_B$) and always with only one offspring, is the case where $\nu(ds)=\mu_B\delta_1(ds)$, where $\delta_1(\cdot)$ is the Dirac measure concentrated in $1$. Here, at (random) times $t_i=\phi_{pi}$ we have $\xi_p(t_i)-\xi_p(t_i-)=1$, where $\xi_p(t-)=\lim_{s\searrow 0}\xi_p(t-s)$.
    \end{example}
    
    \subsection*{Particle death}
    In the heuristic description, we have also considered the death of particles. Rigorously, each living particle $p$ that moves freely (i.e. which is not settled) can die after an exponential random time (independently for each particle $p$) with the rate $\delta_M\in[0,\infty)$, where  $\delta_M=0$ means that the particle cannot 
    (exponentially) die while moving. The heuristic description also considers the dying of a particle after it exits some open set $D\subset E$. Here we recall that we have already included such death by allowing the process $Z$ not  to be conservative. Indeed, a (strong) Markov process $Z$ living on $E$ and killed upon exiting $D$ is again a Markov process living now in $D$.
    
    Further, each settled particle can die after an exponential random time with rate $\delta_S\in [0,\infty]$ (independently for each particle $p$ and independently of the dying mechanism during movement). Here, $\delta_S=0$ means that the particle cannot die while being settled. On the other hand, $\delta_S=\infty$ means that the particle $p$ after it gets settled produces a random number of offsprings (by following the distribution of $\xi_0$) and then immediately dies. Additionally, the settled particle dies after it produces more than $L$ offsprings in total. This is rigorously defined as the time when the compound Poisson process $\xi_p$ crosses the level $L$, i.e. as the exit time from the open set $(-\infty,L)$.

    \begin{remark}
        In cancer, death of particles can occur due to various factors. While settled, cancer cells are under attack of T-cells administered by the immune system. Additionally, upon growing to a particular size, oxygen cannot be diffused to the tumour's centre causing it to enter into a state of hypoxia which, among other consequences, can lead to cancer cell death. On the other hand, moving cancer cells are subject to all kinds of forces and pressures applied to them as they travel with the vascular flow causing them to die before reaching a secondary site. Similarly as for particle settlement, we abstract these underlying biological mechanisms by exponential death rates which in practice can be inferred from localised dynamics.

        This abstraction makes our model not only applicable for modelling the specific case of metastatic spread in cancer but other phenomena with similar global scale dynamics as well, e.g., seed dispersal in agricultural applications, epidemic spread or forest fire spread.
    \end{remark}
    
    \subsection*{Probability space construction}

    Let $(\Omega,\FF, \P)$ be the (product) probability space on which are defined for all $p\in \LL$ the processes $Z_p$, $\theta_p$, $\xi_p$, and the exponentials $\EXP(\delta_M)_p$, $\EXP(\mu_S)_p$ and $\EXP(\delta_S)_p$\footnote{These denote the exponential death while moving, the (exponential) settling time, and the exponential death during settlement, respectively.} and which are all independent, and also independent as a family indexed by $p\in\LL$.

    We can define the metastatic branching process $X=(X(t))_t$ on $(\Omega,\FF, \P)$ by the iterative procedure such that for $t\ge0$ we have
    \begin{align}\label{eq:def-X}
        X(t)&=(\MM(t),\{\wt Z_p(t):p\in \MM(t)\}, \{(\wt\theta_p(t),\wt\xi_p(t)):p\in \MM(t)\}).
    \end{align}
    Here, $\MM (t)$ is, for any $t \geq 0$, a (random) set of living particles at time $t$, the second family contains their positions, and the third family contains the statuses and the number of offsprings that each of the particle in $\MM$ had up to time $t$. We consider the last two families as ordered family, i.e. as vectors rather then sets.
    
    The construction of the process is best understood by its construction in the following steps: Fix $x\in E$ and put $X(0)=(\Root,x,(0,0))$.
    \begin{enumerate}
        \item For $t\in[0,S_{\Root}\wedge D_{\Root})$ we set $X(t)=\big(\Root, Z_{\Root}(t),(0,0)\big)$, where $S_{\Root}=\EXP(\mu_S)_\emptyset$ and $D_{\Root}=\EXP(\delta_M)\wedge T_D$ denote the settlement and the death time of the root particle $\Root$. Here $Z_{\Root}(t)$ moves under the probability measure $\P(\cdot|Z(0)=x)$.
        \item 
        If $S_{\Root}< D_{\Root}$, then in $t\in[
        S_{\Root},(\phi_{\Root 1}+S_{\Root})\wedge (\EXP(\delta_S)+S_{\Root}))$ we have $X(t)=\big(\Root, Z_{\Root}(S_{\Root}),(1,0)\big)$, and in $t=\phi_{\Root 1}+S_{\Root}$, if $\xi_\emptyset(\phi_{\Root1})=k$ and by assuming $\phi_{\Root 1}+S_{\Root}<\EXP(\delta_S)+S_{\Root}$, we have 
        \begin{align*}
            X(t)&=\Big(\{\Root,1,\dots,k\}, \\
            &\quad \{Z_{\Root}(S_{\Root}), Z_1(t-(\phi_{\Root 1}+S_{\Root})),\dots,Z_k(t-(\phi_{\Root 1}+S_{\Root}))\},\\
            &\quad\quad(\{(1,k),\underbrace{(0,0),\dots,(0,0)}_{k\text{ times}}\}\Big).
        \end{align*}
        Here, e.g. the birth time of the particle $1$ is $B_1=S_{\Root}+\phi_{\Root1}$ and it is independent of $Z_1$, $\xi_1$ and $\theta_1$.
        \item We can repeat this procedure, so for  some time $t>0$ we get \eqref{eq:def-X}, where for $p\in \MM(t)$ we have
         \begin{align}\label{def:tilde-Z}
            \wt Z_p(t)\coloneqq Z_p\big((t-B_p)\wedge S_p\big),
        \end{align}
        where $B_p$ denotes the birth time of the particle $p$ (which is independent of $Z_p$ and depends on the information given by $\PAR (p)$), and $S_p$ is the settlement time of the particle $p$. Similarly, $\wt\theta_p(t)\coloneqq \theta_p(t-B_p)$ denotes the status of settlement of the particle $p$ relative to its birth time, while $\wt\xi_p(t)\coloneqq \xi_p\big((t-S_p)\vee0\big)$ is the number of offsprings that $p$ had until time $t$ started counting from the settlement time of the particle $p$.\footnote{Here $a\wedge b=\min\{a,b\}$ and $a\vee b=\max\{a,b\}$}.
    \end{enumerate}

It is clear that we could repeat the same procedure by considering also $X(0)=(\Root,x,(1,k))$, with $k\in \N_0$, i.e. the particle $\Root$ may be already settled, and it already gave birth to $k$ offspings. Then, the new offsprings, if e.g. $\xi_{\Root}(\phi_{\Root 1})=r$, would have labels $\Root(k+1),\dots,\Root(k+r)$. Moreover, we could consider starting even from $X(0)=(q,x,(s,k))$, with $q\in \LL$, $x\in E$, and $s\in \{0,1\}$, $k\in \N_0$, where, then, the offsprings get labels that are in the subtree which has  as the root the label $q$.

By the construction, the process $X$ is a c\`adl\`ag process, and can be written as a functional $X=F(q,x,(s,k),\{Z_p,\theta_p,\xi_p,\EXP(\delta_M)_p,\EXP(\mu_S)_p,\EXP(\delta_S)_p:p\in\LL\})$, under $X(0)=(q,x,(s,k))$. It is easily seen that $X$ lives on the state space $\XX$ for which
    \begin{align}\label{eq:def-statespace}
    x\in \mathcal{X} &\overset{def.}{\iff} x=\Big(K,(z_1,\dots,z_{|K|}), \big((s_1,c_1),\dots,(s_{|K|},c_{|K|})\big)\Big),
    \end{align}
    where $K$ is a finite subset of $\LL$, $z_1,\dots,z_{|K|}\in E$, $s_1,\dots,s_{|K|}\in \{0,1\}$, and $c_1,\dots,c_{|K|}\in \N_0$. We only have finite sets $K$ since the probability that at a given time $t>0$ there is an infinite number of live particles is zero, due to Lemma \ref{lem:FiniteChildren}. Before we prove the lemma, we discuss examples of the known branching processes that our setting covers.

    \begin{example}[Branching Brownian motion]
        It can be seen that the standard dyadic branching Brownian motion (dBBM) is special case of the branching metastatic process by choosing its parameters accordingly. The Markov process $Z$ is of course the standard Brownian motion, and since the moving particles should be split in two, we take $\delta_M=0$, i.e. no death while moving, and $\delta_S=\infty$, $\xi_0\equiv 2$, i.e. death occurs immediately upon settlement and in that moment two offsprings are born. Thus, the reproduction rate of dBBM corresponds to the settlement rate, i.e. $\mu_S=1$ in the standard case.
        
        More generally, in some versions, the  death rate $\delta_M>0$ is also considered, and more general split mechanism as well (i.e. $\xi_0$ is a proper random variable in $\N$).
    \end{example}

    In the next lemma, and in the rest of the article, we impose the following assumption on the expected number of offsprings.
    \begin{assumption}{N}{}\label{as:N}
        It holds that $\E \xi_0<\infty$ and $\sum_{k=1}^{\infty}\nu(\{k\})<\infty$.
    \end{assumption}
    \begin{lemma}\label{lem:FiniteChildren}
	Assume \ref{as:N}. The number of living particles at time $t>0$ of the process $X$, denoted by $|\MM(t)|$, is almost surely finite. Moreover, it holds that 
    \begin{align}\label{eq:1420}
        \E |\MM(t)|\le e^{ C\,t}, \quad t\ge0,
    \end{align}
    where $C= (\mu_S+\mu_B)\big(\E \xi_0+\mu_B^{-1}\sum_{k=1}^\infty k\nu(\{k\})+1\big)$.
    \end{lemma}
    \begin{proof}
        The finite number of particles at each $t>0$ comes directly from the interplay of the metastatic branching process ad the classical theory of multi-type Markov branching processes, and the finiteness of the expectation of $|\MM(t)|$ can be found in \cite[Chapter V]{AthNey}. 
        
        However, to get the bound of the form \eqref{eq:1420}, one can do a simple pairing argument with the one-type Markov branching process as follows.

        Since we want to prove that $|\MM(t)|$ is almost surely finite, it is enough to consider only the case when our moving Markov process $Z$ cannot die on its own (since, if it can die, the number of alive particles will be even smaller).
        Also, we can remove the settlement of the particles from the process and consider the mechanism: each particles freely moves and after an (independent) exponential time with rate $\mu_S+\mu_B$ produces $\xi_0+Y+1$ offsprings and immediatelly dies, where $Y$ has the distribution of the jumps of the compound Poission process $\xi$ (i.e. $Y\sim \nu(\cdot)/\mu_B$) and is independent of $\xi_0$; the offsprings then repeat (independently of each other) the same mechanism.

        The newly obtained process is a classical Markov branching process, as in \cite[Chapter V]{Harris-branching1963}, with the rate of splitting $\mu_s+\mu_B$ and the distribution of offsprings $\xi_0+Y+1$. For brevity, we denote it in the proof by $\overline X$.

        Now we describe the pairing between the  Markov branching process $\overline X$ and the metastatic branching process $X$. In $X$, the particle $\Root$ waits $\EXP(\mu_S)$ before it settles and then produces $\xi_0$ offsprings. We pair the waiting time $\EXP(\mu_S)$ with the splitting time $\EXP(\mu_S+\mu_B)$ of $\Root$ of $\overline X$ so that $\EXP(\mu_S+\mu_B)\le \EXP(\mu_S)$ (e.g. consider the uniform random variable $U$ and put $\EXP(\mu_S)=-\frac{\ln(U)}{\mu_S}$ and $\EXP(\mu_S+\mu_B)=-\frac{\ln(U)}{\mu_S+\mu_B}$). We also pair $\xi_0$ of $\Root$ of $X$ with $\xi_0$ of $\Root$ of $\overline X$ exactly, while $Y$ of $\Root$ of $\overline X$ is left independent. 
        
        Further, since $\Root$ of $\overline X$ produces $\xi_0+Y+1$ offsprings, this "$+1$" particle is looked at as the settled $\Root$ in $X$. Thus, if the settled $\Root$ of $X$ produces new $Y_1$ offsprings after $\EXP(\mu_B)$ time, we pair this $\EXP(\mu_B)$ with the splitting time $\EXP(\mu_S+\mu_B)$ of "$+1$" (so that that $\EXP(\mu_B)\le \EXP(\mu_S+\mu_B)$ in the manner as before), and the new number of offsprings $\xi_0 + Y+1$ of "$+1$" by putting $Y=Y_1$ exactly, while $\xi_0$ is independent. We use such pairing for all particles of $X$.

        This procedure produces two important consequences:
        \begin{itemize}
                \item every particle in the metastatic process $X$ is paired to a particle in the Markov branching process $\overline X$;
                \item for every such particle $p$ in $X$, and every $\omega \in \Omega$, we have that $B_p (\omega,X)\geq {B}_p(\omega,\overline X)$, i.e. the birth time of the particle $p$ in the process $X$ is bigger than the birth time of the paired particle in the process $\overline X$.
            \end{itemize}
        In other words, the number of alive particles at time $t$ of the process $X$, denoted by $N(t)$, is always smaller then the number of alive particles at time $t$ of the process $\overline X$, denoted by $\overline N(t)$, i.e. $N(t)\le \overline N(t)$.

        From the general theory of the Markov branching processes, see e.g. \cite[Theorem Theorem 6.1 in Chapter V]{Harris-branching1963}, it is know that $\overline N(t)$ is finite almost surely and that $\E\overline N(t)= e^{t\,(\mu_S+\mu_B)\E[\xi_0+Y+1]}$, from which the claim follows.
    \end{proof}

	\section{Markov property}\label{s:markov}
	In this section, we prove the Markov property of the constructed metastatic branching process given by Definition~\ref{eq:def-X}, and we keep assuming \ref{as:N}.
    The Markov property of the branching process tracking only the number of particles (i.e., the multi-type process, as explained before) is clear (see, e.g., \cite[Chapter V, Section 7]{AthNey}). Here we briefly discuss the Markov property of our process tracking particle positions: this is a very intuitive property for this process. However, the technicalities involved are instructive to understand the behavior of the process and we believe it is useful to discuss the main steps of the proof. This is done below.
    \begin{theorem}[Markov property]\label{thm:markov}
        The branching metastatic process $X$ is Markov on $\mathcal{X}$ with respect to its natural filtration {$\mathcal{F}_t$, $t \geq 0$}.
    \end{theorem}

    \begin{proof}
        It is clear that
        \begin{equation*}
            \pi_s = \bigcup_{n\geq 1}\bigcup_{0= s_0 < s_1 < \dots < s_n \leq s} \sigma\left(X(s_1), \dots, X(s_n)\right),
        \end{equation*}
        is a $\pi$-system generating the $\sigma$-algebra $\mathcal{F}_s$, so it is enough to prove the Markov property on a finite set of positions $X(s_1), \dots, X(s_n)$. Therefore, for $j\in\N$, and a finite set of indices $0\le s_1\le  s_2, \dots, \le s_j\le s_{j+1}$, we want to show 
        \begin{align}\label{eq:Markov-0}
            &\P\left[f(X(s_{j+1}))\middle|X(s_1), \dots, X(s_j)\right]=\P\left[f(X(s_{j+1}))\middle|X(s_j)\right],
        \end{align}
        for all $f\in \BB_+(\XX)$. Furthermore, \eqref{eq:Markov-0} is enough to prove for $f(K,\mathbf{z},(\mathbf{s},\mathbf{c}))=\1_A(K)\1_B(\mathbf{z})\1_C(\mathbf{s},\mathbf{c})$, with $A\subset \LL$, $B=B_1\times \cdots \times B_{|A|}\subset E^{|A|}$, $C=C_1\times C_2$ where $C_1\subset \{0,1\}^{|A|}$ and $C_2\subset \N_0^{|A|}$. Hence, we need to prove
        \begin{align}
            \begin{split}\label{eq:Markov-1}
            &\P\left(\MM(s_{j+1})\in A, \left\{\wt Z_p(s_{j+1}): p \in \MM(s_{j+1})\right\}\in B, \right.\\
            &\qquad \left.\left\{(\wt\theta_p(s_{j+1}),\wt\xi_p(s_{j+1})):p\in \MM(s_{j+1})\right\}\in C\middle|X(s_1), \dots, X(s_j)\right)
            \end{split}\\
            \begin{split}\label{eq:Markov-2}
            =& \P\left(\MM(s_{j+1})\in A, \left\{\wt Z_p(s_{j+1}): p \in \MM(s_{j+1})\right\}\in B, \right.\\
            &\qquad \left.\left\{(\wt\theta_p(s_{j+1}),\wt\xi_p(s_{j+1})):p\in \MM(s_{j+1})\right\}\in C\middle|X(s_j)\right),
            \end{split}
        \end{align}
        for  $A,B$ and $C$ as above.

        For each particle alive at time $s_{j+1}$, there may exist many of its predecessors alive at time $s_j$. However, only one of them is the closest one in the genealogical tree. Therefore, there is a partition of $\MM(s_{j+1})=\cup_{p\in \MM(s_{j})}\MM_p(s_{j+1})$, where $\MM_p(s_{j+1})$ consists of direct descendants of the particle $p\in \MM(s_j)$ that are alive at time $s_{j+1}$. In other words, between $p\in\MM(s_j)$ and $q\in \MM_p(s_{j+1})$ there does not exists another $p'\in \MM(s_j)$ such that $p'$ is in the middle of the genealogical line between $p$ and $q$. 
        
        Now, we can also, without loss of generality, consider the set $A$ to consist of particles that can indeed be alive (i.e. that can be obtained) by time $s_{j+1}$. Hence, we can make a (random) partition of $A=\cup_{p\in \MM(s_j)} A_p$, where $A_p$ denotes the (random) set of particle labels that can be produced (only/directly in the sense described above) by the particle $p$. Similarly, let $B_p$ and $C_p$ denote the (random) parts of $B$ and $C$, respectively, which correspond to the descendants of the particle $p\in \MM(s_j)$.
        It follows that \eqref{eq:Markov-1} equals to
        \begin{align}
            \label{eq:Markov-3}
                &\P\left(\bigcap_{p\in\MM(s_j)} \left(\MM_p(s_{j+1})\in A_p, \left\{\wt Z_q(s_{j+1}): q \in \MM_p(s_{j+1})\right\}\in B_p, \right.\right.\notag \\
                &\qquad \left.\left.\left\{\wt \theta_q(s_{j+1}),\wt\xi_q(s_{j+1}): q \in \MM_p(s_{j+1})\right\}\in C_p\right)\middle|X(s_1), \dots, X(s_j)\right).
                 \end{align}
        Now note that conditionally on $X(s_j)$, $\dots$, $X(s_1)$, the particle movements after time $s_j$ are independent of each other by the construction of the process. Hence, \eqref{eq:Markov-3} equals to
        \begin{align}
             \begin{split}\label{eq:Markov-4}
                &\prod_{p\in\MM(s_j)} \P\left(\MM_p(s_{j+1})\in A_p, \left\{\wt Z_q(s_{j+1}): q \in \MM_p(s_{j+1})\right\}\in B_p, \right.\\
                &\qquad \left.\left\{\wt \theta_q(s_{j+1}),\wt\xi_q(s_{j+1}): q \in \MM_p(s_{j+1})\right\}\in C_p\middle|X(s_1), \dots, X(s_j)\right).
            \end{split}
        \end{align}
        Now we pay attention to the terms inside of \eqref{eq:Markov-4}. Recall that the process $X$ started in $x$ is a functional of random quantities $F(x, \{\EXP(\delta_M)_p,\EXP(\mu_S)_p,\EXP(\delta_S)_p,Z_p,\xi_p:p\in\LL\})$, and that, conditionally on $X(s_j)$, $\dots$, $X(s_1)$, the process 
        \begin{align}
            \begin{split}\label{1207process}
            t\mapsto (\MM_p(t+s_{j}),&\{\wt Z_q(t+s_{j}): q \in \MM_p(t+s_{j+1})\},\\
            &\{\wt \theta_q(t+s_{j}),\wt\xi_q(t+s_{j}): q \in \MM_p(t+s_{j})\})
            \end{split}
        \end{align}
        starts from $\wt x=(\{p\},\wt Z_p(s_j),(\wt\theta_p(s_j),\wt\xi_p(s_j))$, and it is in law equal to the functional $F(\wt x,\{Exp(\delta_M)_p,Exp(\mu_S)_p,Exp(\delta_S)_p,Z_p,\xi_p:p\in\LL\})$. In other words, the process in \eqref{1207process} can be constructed in the same way as the process $t\mapsto X(t)$ started from the point $x=(\{p\},\wt Z_p(s_j),(\wt\theta_p(s_j),\wt\xi_p(s_j))$,
        so the term inside the product in \eqref{eq:Markov-4},
        for $p\in \MM(s_j)$, equals to
        \begin{align}
        \begin{split}\label{eq:Markov-homog}
            &\P\left(\MM(s_{j+1}-s_j)\in A_p, \left\{\wt Z_q(s_{j+1}-s_j): q \in \MM(s_{j+1}-s_j)\right\}\in B_p, \right.\\
                &\qquad \left.\left\{\wt \theta_q(s_{j+1}-s_j),\wt\xi_q(s_{j+1}): q \in \MM(s_{j+1}-s_j)\right\}\in C_p\right.\\
                &\left.\qquad\qquad\middle|X(0)=\big(\{p\},\wt Z_p(s_j),(\wt\theta_p(s_j),\wt\xi_p(s_j)\big)\right).
        \end{split}
        \end{align}
        Moreover, note that by changing from conditioning in \eqref{eq:Markov-homog} from $\big(\{p\},\wt Z_p(s_j),(\wt\theta_p(s_j),$ $\wt\xi_p(s_j)\big)$ to $X(s_j)$ we add no relevant information, so \eqref{eq:Markov-homog} becomes 
        \begin{align}
        \begin{split}\label{eq:Markov-homog-2}
            &\P\left(\MM_p(s_{j+1}-s_j)\in A_p, \left\{\wt Z_q(s_{j+1}-s_j): q \in \MM_p(s_{j+1}-s_j)\right\}\in B_p, \right.\\
                &\qquad \left.\left\{\wt \theta_q(s_{j+1}-s_j),\wt\xi_q(s_{j+1}): q \in \MM_p(s_{j+1}-s_j)\right\}\in C_p\middle|X(0)=X(s_j)\right).
        \end{split}
        \end{align}
        Thus, \eqref{eq:Markov-4} is equal to
        \begin{align}
            \begin{split}\label{eq:Markov-5}
                &\prod_{p\in\MM(s_j)} \P\left(\MM_p(s_{j+1}-s_j)\in A_p, \left\{\wt Z_q(s_{j+1}-s_j): q \in \MM_p(s_{j+1}-s_j)\right\}\in B_p, \right.\\
                &\qquad \left.\left\{\wt \theta_q(s_{j+1}-s_j),\wt\xi_q(s_{j+1}): q \in \MM_p(s_{j+1}-s_j)\right\}\in C_p\middle|X(0)=X(s_j)\right)
            \end{split}\\
            \begin{split}\label{eq:Markov-final0}
                &=\P\left(\bigcap_{p\in\MM(s_j)} \left(\MM_p(s_{j+1}-s_j)\in A_p, \left\{\wt Z_q(s_{j+1}-s_j): q \in \MM_p(s_{j+1}-s_j)\right\}\in B_p, \right.\right.\\
                &\qquad \left.\left.\left\{\wt \theta_q(s_{j+1}-s_j),\wt\xi_q(s_{j+1}-s_j): q \in \MM_p(s_{j+1}-s_j)\right\}\in C_p\right)\middle| X(0)=X(s_j)\right).
            \end{split}\\
            \begin{split}\label{eq:Markov-final}
                &=\P\left(\MM(s_{j+1}-s_j)\in A, \left\{\wt Z_q(s_{j+1}-s_j): q \in \MM(s_{j+1}-s_j)\right\}\in B_, \right.\\
                &\qquad \left.\left\{\wt \theta_q(s_{j+1}-s_j),\wt\xi_q(s_{j+1}-s_j): q \in \MM(s_{j+1}-s_j)\right\}\in C\middle|X(0)=X(s_j)\right).
            \end{split}
        \end{align}
        where in the second equality we again used the conditional independence of the future trajectories, while the last one is due to the partitioning of $A$.

        This shows the time-homogeneous Markov property since the right hand side of \eqref{eq:Markov-final} depends only on $X(s_j)$ and the time increment $s_{j+1}-s_j$.
    \end{proof}

	\section{Feller property}\label{s:feller}
       Recall that each time-homogeneous Markov process $Y$ on some state-space $S$ {can be associated with} a semigroup of operators $(P_t)_t$ on $\BB_b(S)$ by the formula $P_t f(x)=\E[f(Y_{t+s})|Y_s=x]$, where the term does not depend on $s>0$ by time-homogeneity. So, we may also write $P_tf(x)=\E[f(Y_t)|Y_0=x]$, {although the original construction of the process $Y$ may have been done in the canonical sense, i.e., on $D[0, +\infty)$ with respect to a family of measures $\P^x$, in such a way that $Y_0=x$ a.s. under $\P^x$}. Hence, also our metastatic process $X$ generates the semigroup of operators $P=(P_t)_t$.

	In this section, we build on the results from the previous section to prove that the metastatic branching process $X$ is not only a Markov process but also a Feller process (the corresponding metric space will be constructed in the following paragraphs). This will be possible under a slightly different representation of the metastatic process $X$, the one which does not track all labels of the particles together with their statuses, but the one that tracks only the number of moving and the number of settled particles, { i.e., a kind of multi-type process in the sense as in \cite{AthNey}.} We will show that the latter satisfies the Feller property (in Theorem \ref{t:feller}) and also we will show that the metastatic process tracking the positions (as constructed before) cannot satisfy the Feller property under the metric that seems the only reasonable for the original metastatic process (see, in particular, Remark \ref{r:notFeller}).

    {We stress that the Feller property is crucial as it permits us to determine the generator and (subset of) its domain by computing the limit pointwise (see Section \ref{s:generator} below for the details).}

    We consider the modification of the metastatic branching process, denoted by $\wt X$, with the values in the space denoted by $\wt \XX$ with
    \begin{align}\label{eq:def-statespace-No2}
        x\in \wt\XX \overset{\text{def.}}{\iff} x=\Big(n_m^x,n_s^x,(z_p)^{n_m^x+n_s^x}_{p=1}\Big),
    \end{align}
    where $n_m^x\in \N_0$, $n_s^x\in \N_0$ and $(z_p)^{n_m^x+n_s^x}_{p=1}\in E^{n_m^x+n_s^x}$, with the following meaning. The first coordinate represents the number of freely moving particles, the second coordinate represents the number of settled particles, and the third coordinate is a vector of the positions of the moving and settled particles. Here, we adopt the convention that the first $n_m^x$ terms of the last vector correspond to the moving particles, while the last $n_s^x$ terms correspond to the settled particles. Also, the case $n_m^x=n_s^x=0$ corresponds to the scenario where the process dies and then the third coordinate is at the additional point $\partial$ called the cemetery, and this is clearly an absorbing state. We use the notation
    \begin{align}\label{eq:modified-X-repr}
        \wt X(t)=\left(N_m(t),N_s(t),(\wt Z_p(t))_{p=1}^{N_m(t)+N_s(t)}\right),
    \end{align}
    where $\wt Z_p(t)$ is given in the same way as in \eqref{def:tilde-Z}. Finally, we describe the changes in the branching process by adopting the following conventions.
    \begin{itemize}
        \item If a new particle is born at time $t$, we put its position as $N_m(t)+1$--st coordinate in the vector of the positions (and increase $N_m(t)$ by one).
        \item If a moving particle is being settled at time $t$, we move its position to the $N_m(t)+N_s(t)$--st position (and decrease $N_m(t)$ and increase $N_s(t)$, both by one).
    \end{itemize}      
    
   {In the same spirit as Theorem \ref{thm:markov}, it is not hard to see that this modified process $\wt X$ is a time-homogeneous Markov process. Therefore,} it is associated with the semigroup of operators, denoted here by $P=(P_t)_t$. Since it is also a càdlàg process, from now on we consider it defined as a canonical process in the sense of \cite{Revuz}, and move to the customary notation $\mathds{E}^x [f(\wt X(t))]=P_t f(x)=\E[f(\wt X({t+s}))|\wt X(s)=x]$, and similarly $\mathds{P}^x(\cdot)=\P(\cdot|\wt X(0)=x)$. 
    
    To prove that $\wt X$ is a Feller process means that its semigroup $P=(P_t)_t$ satisfies the Feller property on the space $C_0(\wt\XX)$. In other words, this means that on the state-space $\wt\XX$ we need to define a (natural/intuitive) metric $d_{\wt\XX}(\cdot,\cdot)$, {understand the space} $C_0(\wt\XX)$ (the space of continuous functions on $\wt\XX$, vanishing at infinity, with respect to the metric $d_{\wt\XX}$), and prove the Feller conditions:
    \begin{itemize}
        \item For all $t>0$, $P_t\left(C_0\big({\wt\XX}\big)\right)\subset C_0\big({\wt\XX}\big)$,
        \item For all $f\in C_0\big({\wt\XX}\big)$, $\lim_{t\searrow0} P_tf=f$ in $C_0\big({\wt\XX}\big)$.
    \end{itemize}
    
    \subsection*{Metric on ${\wt\XX}$ and space $C_0\big({\wt\XX}\big)$}

    Recall that the space ${\wt\XX}$ essentially consists of three parts, see \eqref{eq:def-statespace-No2}, where the first and the second coordinate tell the number of moving/settled particles, while the third part tells the positions of the particles. Therefore, we define the metric $d_{\wt\XX}$ on ${\wt\XX}$ by using the discrete metric on the first and the second part, and on the third part, we use the intrinsic distance in $E$ of the alive particles.
    
    To this end, let $x,y\in {\wt\XX}$ be written as 
    \begin{align}
        x&=\Big(n_m^x,n_s^x,(z^x_p)^{n_m^x+n_s^x}_{p=1}\Big)\\
        y&=\Big(n_m^y,n_s^y,(z^y_p)^{n_m^y+n_s^y}_{p=1}\Big),
    \end{align}
    and define the function $d_{\wt\XX} : {\wt\XX}\times {\wt\XX} \to [0,\infty)$ by
    \begin{align}\label{eq:Metric}
		d_{\wt\XX}(x,y) &\coloneqq \max\Big\{\1_{\{n_m^x\neq n_m^y\}},\1_{\{n_s^x\neq n_s^y\}}\Big\} \\
        &\qquad\qquad+\1_{\{n_m^x= n_m^y\}}\1_{\{n_s^x= n_s^y\}}\left(\sum_{p=1 }^{n_m^x+n_s^x} d_E(z^x_{p}, z^y_{p})\right)\wedge 1,
	\end{align}
	where we recall that $d_E$ denotes the metric on $E$ (supporting the generic Feller process $Z$ relative to $C_0(E)$).

    \begin{lemma}\label{lem:Metric}
        The function $d_{\wt\XX}$  defines a metric on $\wt\XX$. Moreover, ${\wt\XX}$ equipped with the metric $d_{\wt\XX}$ is a locally compact and separable metric space.
    \end{lemma}
    \begin{proof}
        It is an easy exercise to prove that $d_{\wt\XX}$ is a metric on ${\wt\XX}$.
        
        To prove local compactness, it is enough to prove that for each $x\in {\wt\XX}$, there is $\varepsilon>0$ such that $\overline{B_{\wt\XX}(x,\varepsilon)}$ is compact.  Take $x=\Big(n_m^x,n_s^x,(z^x_p)^{n_m^x+n_s^x}_{p=1}\Big)\in {\wt\XX}$, so for each $z_p^x$ there is $\varepsilon_p^x>0$ such that $\overline{B_E(z_p^x,\varepsilon_p^x)}$ is compact (since  $E$ is locally compact metric space). Take now $\varepsilon=\min\{\varepsilon_p^x:p\in\{1,\dots,n_m^x+n_s^x\}\}\wedge \frac12$. Now $B_{\wt\XX}(x,\varepsilon)$ consists of the points $y\in {\wt\XX}$ which have the same number of moving particles as well as the same number of settled particles, while the distances between corresponding particles are less than $\varepsilon$. Since we deal with a metric space ${\wt\XX}$, we use a sequential compactness argument that immediately shows that $\overline{B_{\wt\XX}(x,\varepsilon)}$ is compact.

        To prove separability of ${\wt\XX}$, it is enough to observe that the number of possible combinations of the number of moving particles together with the number of settled particles is countable, and that the space $E$ is separable, so the claim easily follows.
    \end{proof}
    
    \begin{remark}\label{rem:Compactness}
        Here we note that a compact subset $\KK\subset {\wt\XX}$, although it may have an uncountable number of elements, it necessarily contains only { a finite number of different numbers of moving particles}, and for each of these numbers of moving particles, it contains only a finite number of different numbers of settled particles. Namely, if $x,y\in {\wt\XX}$ have a different number of moving particles, i.e. $n_m^x\neq n_m^y$, then $d_{\wt\XX}(x,y)=1$, so if  $\KK$ would contain infinitely many numbers of moving particles, not all open covers would have a finite subcover, which is a contradiction with compactness. A similar argument can be applied to show that the set containing the numbers of moving particles must also be finite.
    \end{remark}

    The space $C_0({\wt\XX},d_{\wt\XX})=C_0\big({\wt\XX}\big)$ now becomes the classical space of continuous functions vanishing at infinity, with respect to the metric $d_{\wt\XX}$: i.e. $f\in C_0\big({\wt\XX}\big)$ if it is continuous from $({\wt\XX},d_{\wt\XX})$ to $\R$, and if for every $\varepsilon>0$, there exists a compact set $\KK\subset {\wt\XX}$ (with respect to the metric $d_{\wt\XX}$) such that $|f(x)|<\varepsilon$, $x\in {\wt\XX}\setminus \KK$.

    To be able to do all technical calculations rigorously, {in addition to the assumption \ref{as:N},} we impose the following assumption of the metastatic process.
    \begin{assumption}{A}{}\label{as:A}
        The Markov process $Z$ is a conservative Feller process on $E$ with the semigroup $(T_t)_{t}$\footnote{In other words, the semigroup $(T_t)_t$ is a strongly continuous semigroup on $C_0(E,d_E)$, {equipped with the sup-norm}, and it holds that $T_t\1=\1$.}. The particles of the metastatic process die only by the exponential mechanism, with the rates $\delta_M\ge 0$ and $\delta_S\in [0,\infty)$, and the settlement happens only via the exponential mechanism with the rate $\mu_S>0$.
    \end{assumption}

    \begin{theorem}[Feller property of the metastatic branching process]\label{t:feller}
        Under \ref{as:N} and \ref{as:A}, the metastatic branching process $\wt X$ is a Feller process on $({\wt\XX},d_{\wt\XX})$.
    \end{theorem}
    \begin{proof}
    We start by proving that $P_tf \in C_0\big({\wt\XX}\big)$ for $f\in C_0\big({\wt\XX}\big)$ and $t>0$. First, we prove vanishing at infinity. 
    
    Take $f\in C_0\big({\wt\XX}\big)$ and $t>0$, and fix $\varepsilon>0$. Then there exists a compact $\KK\subset {\wt\XX}$ such that $|f(x)|<\varepsilon$ for all $x\notin\KK$. By Remark~\ref{rem:Compactness}, there is a finite number of combinations of the number of moving and settled particles contained in $\mathcal{K}$; here we denote them by $K_i\coloneqq(n^{(i)}_m,n^{(i)}_s)$, for $i=1,\dots,N$, where $N=N(f,\varepsilon)\in \N$ is fixed. Additionally, there exists a compact $\mathfrak{K}\subset E$ such that $|f(n^{(i)}_m,n^{(i)}_s,\cdot)|<\varepsilon$ on $A_{i}^c\coloneqq \big(\otimes_{p=1}^{n^{(i)}_m+n^{(i)}_s}\mathfrak{K}\big)^c$, for all $i=1,\dots,N$.
    In fact, without loss of generality, we may assume that $\mathcal{K}={\cup_{i=1}^N} \{n_m^{(i)}\}\times\{n_s^{(i)}\}\times A_i$, and where we note that the union is disjoint.

    We have for all $x=(n_m,n_s,(z_p)_{p=1}^{n_m+n_s})\in {\wt\XX}$
    \begin{align}
        |P_t f(x)|&\le \varepsilon \P^x(\wt X(t)\notin \mathcal{K})\label{1702b}\\
        &\qquad+
        \sum_{i=1}^{N}\E^x\left[f(\wt X(t));\wt X(t)\in\{n^{(i)}_m \}\times\{n^{(i)}_s\}\times A_i\right].\label{1702a}
    \end{align}
    Since the term \eqref{1702b} is less then $\varepsilon$ for all $x\in {\wt\XX}$, we proceed by finding a compact $\wt{\mathcal{K}}$ such that the term \eqref{1702a} is less then $\varepsilon$ for all $x\notin\wt{\mathcal{K}}$.
    
    First, let us see which configuration of starting numbers of moving and settled particles could indeed produce the configuration pairs $K_i$, $i=1,\dots,N$. It is immediately clear that if we start from $x_0=(n_m,n_s,(z_p)_{p=1}^{n_m+n_s})$ with $n_m+n_s\ge 1$, then we can obtain the pairs $K_i$ in the future by having the right amount of settlements, deaths and births, implying that there are infinitely many of possible predecessors to each of $K_i$. However, by Lemma~\ref{lem:FiniteChildren} it is improbable that too many events (settlements, births, deaths) happen  in the time interval $[0,t]$. More precisely, if we start from $x_0=(n_m+n,n_s+n,(z_p)_{p=1}^{n_m+n_s+2n})$ and we want to reach $x=(n_m,n_s,(z_p)_{p=1}^{n_m+n_s})$ at time $t>0$, there must have been at least $n$ deaths of the settled particles (at least $n$ of those that are settled at time $0$), and at least $n$ deaths or settlements of moving particles (at least $n$ of those that are moving at time $0$). Hence,
    \begin{align}
        \begin{split}\label{eq:death-birth-changes}
            \P^{x_0}&\left(\wt X(t)\in \{n_m\}\times \{n_s\}\times E^{n_m+n_s}\right)\\&\qquad\leq \binom{n_m+n}{n}\left(1-e^{-(\mu_S+ \delta_M) t}\right)^{n} \cdot\binom{n_s+n}{n}\left(1-e^{-\delta_S t}\right)^{n},
        \end{split}
    \end{align}
    since the probability that a (fixed) moving particle dies or settles before time $t>0$ equals
    \begin{align}
    \begin{split}\label{aux1}
        \P&(\text{living particle $Z$ dies or settles before time $t$})\\
        &\qquad=1-\P(\EXP(\mu_S)>t,\EXP(\delta_M)>t)=1-e^{-(\mu_S+ \delta_M) t},
    \end{split}
    \end{align}
    where we recall that $\EXP(\mu_S)$ and $\EXP(\delta_M)$ denote the exponentially distributed random variables corresponding to the settlement and the death while moving.

    By using the assumption \ref{as:A}, the expression  \eqref{eq:death-birth-changes} is arbitrarily small  for $n$ big enough.
    Therefore, for $\varepsilon>0$ from the beginning of the proof, there exists $n_0=n_0(N,\varepsilon)\in \N$ such that for all $x\in {\wt\XX}$ of the form $x=(n,n,(z_p)_{p=1}^{2n})$ where $n> n_0$ it holds that the term \eqref{1702a} is bounded from above by $\varepsilon$.

    Now we focus on the starting positions $x=(n_m,n_s,(z_p)_{p=1}^{n_m+n_s})$ where $\max\{n_m,n_s\}\le n_0$. By Lemma \ref{ap:l:feller-escaping}, take  an another compact $\widetilde{\mathfrak{K}}\subset E$, depending only on $t$, $N$ and $\varepsilon>0$, such that the generic Feller particle $Z$ (the one which drives the moving mechanism of the branching process) satisfies for all $z\notin \wt{\mathfrak{K}}$ and all $s\in [0,t]$
    \begin{align}\label{eq:feller-escaping}
        \P(Z(s)\in\mathfrak{K}|Z(0)=z)<\frac{\varepsilon}{N}.
    \end{align}
    Hence, the term \eqref{1702a} if $z_p\in \widetilde{\mathfrak{K}}^c$, for all $p=1,\dots,n_m+n_s$, is bounded from above by $\varepsilon$, since the living particle at time $t>0$ can be looked at as a single particle that moved from time $0$ till time $t$ as the Feller process $Z$, where we erase some intervals of constancy due to the settlements of the parent particles (which depend on the independent exponential mechanism).

    In other words, we found the compact $\wt{\mathcal{K}}=\cup_{n_m=1}^{n_0}\cup_{n_s=1}^{n_0}\{n_m\}\times\{n_m\}\times \wt{\mathfrak{K}}^{n_m+n_s}$ such that for all $x\notin\wt{\mathcal{K}}$ the term \eqref{1702a} is less then $3\varepsilon$. This finishes the proof of vanishing at infinity.

    Now we show that $P_tf$ is locally continuous. Take $x_k\to x$ in ${\wt\XX}$. This means that there exists $k_0\in\N$  such that for all $k\ge k_0$ the number of moving and the number of settled particles in $x_k$ and $x$ are the same. Thus, without loss of generality, we may assume that $x_k=(n_m,n_s,(z^{(k)}_p(t))_{p=1}^{n_m+n_s})$, i.e., the sequence does not depend on the number of moving and settled particles.
    
    We can couple the processes $\wt X(t)|\wt X(0)=x$ and $\wt X(t)|\wt X(0)=x_n$ such that the death, the settling and the reproduction times are exactly the same (since by \ref{as:A} we assume only the exponential mechanisms of death and settlements). This means that the final positions of particles can differ only due to the positions of the starting particles (and not on the settling/reproduction times). {Here we used again that each particle can be looked at time $t$ as a continuation of some original particle that existed at time $0$ that moved like  the generic Feller process $Z$} where we erase some intervals of constancy from the time horizon that are due to the settlement of the parent particles. Hence, we have
    \begin{align*}
        &P_t f(x_n)=\E^{x_n}[f(\wt X(t))]\\&
        =\E^x[T^{\otimes (N_m(t)+N_s(t))}_{L_1(t), \dots, L_{N_m(t)+N_s(t)}(t)} \wt f_{N_m(t),N_s(t)}\big( Z_1(0), \dots, Z_{N_m(t)+N_s(t)}(0)\big)].
    \end{align*}
    Here $\wt f_{k,l}(\cdot)=f(k,l,\cdot)$, while $L_i(t)$, $i=1,\dots,N_m(t)+N_s(t)$, denote the living times of respective particles started from their birth time until time $t$ (for moving particles) or until their settlement time (for settled particles). 
    
    Now we want to use the dominated convergence theorem to show local continuity of $P_tf$, and we will use Lemma  \ref{lem:TensorProductFellerSemigroup} to do so. We need to show that the times $L_i(t)$ and the birth positions of the alive particles at time $t$ are convergent with respect to the starting positions as $x_n\to x$. Since we use the coupling mechanism, the times $t\mapsto L_i(t)$ are always the same regardless of the starting positions, so they are convergent with respect to the starting positions as $x_n\to x$. We proceed by inductively showing that the positions of the finishing particles at their birth times are convergent with respect to the starting positions as $x_n\to x$.

    Assume first that we have $\omega\in \Omega$ such that there have been no settlements and no births in $[0,t]$. This means that the final living particles are the same as the original ones, and 
    \begin{align}
        &T^{\otimes (N_m(t)+N_s(t))}_{L_1(t), \dots, L_{N_m(t)+N_s(t)}(t)} \wt f_{N_m(t),N_s(t)}\big( Z_1(0), \dots, Z_{N_m(t)+N_s(t)}(0)\big)\\
        &\qquad=T^{\otimes (n_m+n_s)}_{\underbrace{t,\dots,t}_{n_m\text{ times}},0,\dots, 0} \wt f_{n_m,n_s}\big( z_1, \dots, z_{n_m}, z_{n_m+1},\dots, z_{n_m+n_s}\big)
    \end{align}    which is continuous in $x$ by Lemma \ref{lem:TensorProductFellerSemigroup}. If we now take $\omega\in \Omega$ such that only one additional birth event happens, by the particle $n_m< j\le n_m+n_s$. Then the newborn particle has its starting position $z_j$ with the living time $t-E_j$ where $E_j\sim Exp(\mu_B)$. However, birth times are driven by the independent exponential mechanism, so they can be conditionally integrated:
    \begin{align*}
        &T^{\otimes (N_m(t)+N_s(t))}_{L_1(t), \dots, L_{N_m(t)+N_s(t)}(t)} \wt f_{N_m(t),N_s(t)}\big( Z_1(0), \dots, Z_{N_m(t)+N_s(t)}(0)\big)\\
        &=\int_0^t \sum_{k=1}^\infty T^{\otimes (n_m+k+n_s)}_{\underbrace{t,\dots,t}_{n_m\text{ times}},\underbrace{(t-s),\dots,(t-s)}_{k\text{ times}},0,\dots, 0} \wt f_{n_m+k,n_s}\big( z_1, \dots, z_{n_m},\\
        &\hspace{16em}\underbrace{z_j,\dots,z_j}_{k\text{ times}},z_{n_m+1},\dots, z_{n_m+n_s}\big)\P(E_j\in ds).
    \end{align*}
    Hence, also in this set of $\omega\in \Omega$, the tensor semigroup is continuous with respect to its argument. 
    Similarly, assume that $\omega\in \Omega$ is such that only one settlement happens, with no other changes in $[0,t]$ and that this is the particle $i$, $1\le i\le n_m$, that settles (with $\xi_0$ immediate births). Then
     \begin{align*}
        &T^{\otimes (N_m(t)+N_s(t))}_{L_1(t), \dots, L_{N_m(t)+N_s(t)}(t)} \wt f_{N_m(t),N_s(t)}\big( Z_1(0), \dots, Z_{N_m(t)+N_s(t)}(0)\big)\\
        &=\sum_{k=0}^\infty q_k\int_0^t\int_E T^{\otimes (n_m+k+n_s)}_{\underbrace{t,\dots,t}_{n_m-1\text{ times}},\underbrace{t-s,\dots,t-s}_{k\text{ times}},0,\dots, 0,(t-s)} \wt f_{n_m-1+k,n_s+1}\big( z_1, \dots,z_{i-1},z_{i+1},\dots z_{n_m},\\
        &\hspace{7em}\underbrace{z,\dots,z}_{k\text{ times}},z_{n_m+1},\dots, z_{n_m+n_s},z\big)\P(Z(t-s)\in dz|Z(0)= z_i)\P(E_j\in ds),
    \end{align*}
    where we recall $q_k=\P(\xi_0=k)$.
    Now we can use the weak continuity of measures $\P(Z(t-s)\in dz|Z(0)= z_i)$ with respect to the starting position (since the generic process $Z$ is Feller, and since the tensor semigroup is also continuous in the sense of Lemma \ref{lem:TensorProductFellerSemigroup}) to conclude  continuity (by also using the dominated convergence theorem) for all such $\omega\in \Omega$.
    
    {Now, it is clear that  a similar argument can be applied, in general, to all $\omega\in \Omega$ with only a finite number of events (settlements, births, deaths) happened in time $[0,t]$. However, this is an almost sure event by Lemma }\ref{lem:FiniteChildren}. Therefore, we conclude by using the dominated convergence theorem that $P_tf(x_k)\to P_tf(x)$, i.e., $P_tf$ is locally continuous.

    It remains to show that $P_t f(x)\to f(x)$, as $t\to0$, for all $x\in {\wt\XX}$, which by \cite[Lemma 1.4]{LevyMatters3} implies the strong continuity. Take $f\in C_0\big({\wt\XX}\big)$ and $x=(n_m,n_s,(z_p)_{p=1}^{n_m+n_s})\in {\wt\XX}$. Since $f$ is bounded and continuous, we can apply the dominated convergence theorem to prove
        \begin{align}
            P_t f(x)=\E^x[f(X(t))]\to f(x),
        \end{align}
        because, almost surely, the configuration of moving and settled particles at small time $t$ becomes the same as at time 0, while the moving particles move with right-continuous trajectories.
    \end{proof}

    \begin{remark}
        We can also prove the Feller property of $X$ by abandoning some of the assumptions of \ref{as:A}.
        
        For example, the parameter $\delta_S=\infty$ can also be considered, and this is the case where there are no settled particles, i.e., $N_s(t)=0$ for all $t\ge 0$. The proof of the Feller property in this case  remains almost the same, since we need to focus on the process with values in $\cup_{n=0}^\infty\{n\} \times E^{n} $.
        
        On the other hand, if the rate $\mu_S=0$, then the particles never settle, and we do not have a branching process, but in essence just the (killed) Feller process $Z$, so $\mu_S>0$ is indispensable if we want to keep the branching mechanism. 
    \end{remark}

    \begin{remark}[{On the non-Feller property of $X$}]\label{r:notFeller}
     {It is interesting that the branching process tracking the particles' positions, as described in \eqref{eq:def-X}, cannot be a Feller process under what appears to be the most reasonable metric on the space $\XX$ as described in \eqref{eq:def-statespace}.}
        Indeed, in \eqref{eq:def-statespace} we have essentially three components: a set $K$ denoting the labels of the living particles and as such should be dealt with the discrete metric; the positions of the living particles $(z_1,\dots, z_{|K|})$ which should be dealt with the inherited metric on the space $E$; and the set of statuses $\big((s_1,c_1),\dots,(s_{|K|},c_{|K|})\big)$ which should be dealth with the discrete metric on the statuses $s_i$, and on the number of offsprings $c_i$'s we should have either the discrete metric or the difference in number of offsprings. In other words, the metric should be (or should be equivalent to)
        \begin{align*}
             d_\XX(x,y)&=\max\Big\{\1_{\{K^x\neq K^y\}},\1_{\{\mathbf{s}^x\neq\mathbf{s}^y\}},\1_{\{\mathbf{c}^x\neq\mathbf{c}^y\}}\Big\} \\
        &\qquad\qquad+\1_{\{K^x{=} K^y\}}\1_{\{\mathbf{s}^x{=}\mathbf{s}^y\}}\1_{\{\mathbf{c}^x{=}\mathbf{c}^y\}}\left(\sum_{p\in K_x }d_E(z^x_{p}, z^y_{p})\right)\wedge 1,
        \end{align*}
        where we used the {obvious} notation $x=(K_x,\mathbf{z}^x, \mathbf{s}^x,\mathbf{c}^x)$. It can be easily seen that $d_\XX$ is indeed a metric on $\XX$ {making it a locally compact and separable metric space.}
        
        However, the place where the Feller property falls apart is in the vanishing at infinity of $P_tf$, for $f\in C_0(\XX, d_\XX)$ and $t>0$. Indeed, by repeating the proof of Theorem \ref{t:feller}, we arrive at the corresponding paragraph before \eqref{eq:death-birth-changes}, where we need to prove that only final number of starting label sets can indeed produce (with high enough probability) the labels $K_i$, $i=1,\dots,N$, those on which it is possible that $|f|\ge \varepsilon$, for some prefixed 
        $\varepsilon>0$. Unfortunately, {the label set, e.g.,  $K=\{(1)\}$ can be obtained from all  $\{(1),(k)\}$, $k\ge 2$, by just one death, which breaks the needed property of vanishing at infinity.}

        To make everything more precise, suppose that the generic Feller movement $Z$ is just the one-dimensional standard Brownian motion, and all the other mechanisms are exponential ones. Take now $f\in C_0(\XX, d_\XX)$ such that $f\equiv 0$ except for $f(\{1\},\cdot,\{0\}\times \{0\})=\wt f(\cdot)\in C_0(\R)$, e.g. $\wt f(x)=e^{-x^2}$. Now for e.g. $t=1$, and all $x_k=(\{(1),(k)\},(0,0),\{(0,0), (0,0)\})$, for $k\ge 2$, we have
        \begin{align*}
            P_tf(x_k)=\E^{x_k}[f(X_1)]\ge (1-e^{\delta_M})e^{-\delta_S }e^{-(\delta_M+\delta_S) }\E^0[\wt f(Z_1)]\eqqcolon C(\delta_M,\delta_S,\wt f)>0,
        \end{align*}
        where the first two exponentials say that the particle labeled with $(k)$ died before $t=1$ but did not settle before $t=1$, while the third exponential says that the particle labeled $(1)$ did not die nor settle before $t=1$. In other words, $P_tf$ does not vanish at infinity since we found one sequence $(x_k)_k$ which cannot be contained in any compact set, and for which $P_tf(x_k)$ does not go to 0 as $k$ goes to infinity.
    \end{remark}

	\section{Generator}\label{s:generator}
    {With the Feller property at hand, we can move to the next goal. In this section, we derive} the infinitesimal generator of the metastatic branching process $\wt X$. We keep the notation and the representation of the process as in Section \ref{s:feller}, i.e.  $\wt X$ is given as in \eqref{eq:modified-X-repr}.

    The infinitesimal generator of some strongly continuous semigroup $(P_t)_t$ on some Banach space $(\BB,\|\cdot\|)$ is given as
    \begin{align}\label{eq:generator}
        \mathcal{A}f=\lim_{t\to 0}\frac{P_tf-f}{t},\quad f\in \DD(\mathcal{A}),
    \end{align}
    where the limit is considered in the strong sense of the Banach space $(\BB,\|\cdot\|)$, and $\DD(\mathcal{A})$ is the domain of the operator $\mathcal{A}$, i.e. the space of all functions in $\BB$ for which the limit in \eqref{eq:generator} exists.

    From the general theory of the Feller processes/semigroups (i.e., on $C_0(E)$), it is known that to obtain the infinitesimal generator, one has to study only pointwise limits, see \cite[Theorem 1.33]{LevyMatters3}, and this is also our approach to obtaining the infinitesimal generator of the branching metastatic process.

    Recall that, under \ref{as:A}, the generic process $Z$ is Fellerian, so it is associated with the strongly continuous semigroup $(T_t)_t$ on $C_0(E)$, and we denote its infinitesimal generator by $(G,\DD(G))$. In the following theorem, for a suitable function $f\in C_0\big(\wt \XX\big)$, we denote by $G_{z_i}f(x)$ the act of the generator on the function $f$ with respect to the position $z_i$ in $x=(n_m,n_s,(z_p)_{p=1}^{n_m+n_s})$.

    \begin{theorem}\label{t:generator}
        Under \ref{as:N} and \ref{as:A}, the infinitesimal generator $(\AA,\DD(\AA))$ of the metastatic branching process $\wt X$ on $C_0\big(\wt \XX,d_{\wt \XX}\big)$ is given by
        \begin{align*}
            &\AA_{|{\mathfrak{D}}} f(n_m,n_s,(z_p)_{p=1}^{n_m+n_s})=\\
            &\sum_{i=1}^{n_m}\delta_M \Big(f\big(n_m-1,n_s,(z_1,\dots,z_{i-1},z_{i+1},\dots,z_{n_m+n_s})\big)-f(n_m,n_s,(z_p)_{p=1}^{n_m+n_s})\Big)\\
            &+\sum_{i=1}^{n_m}\sum_{k=0}^\infty q_k\mu_S\Big(f\big({n_m-1+k},n_s+1,(z_1,\dots,z_{i-1},z_{i+1},\dots,z_{n_m},\underbrace{z_i,\dots,z_i}_{\text{$k$  times}},z_{n_m+1},\dots, z_{n_m+n_s},z_{i})\big)\notag\\
            &\hspace{20em}-f(n_m,n_s,(z_p)_{p=1}^{n_m+n_s})\Big)\\
            &+\sum_{j=n_m+1}^{n_m+n_s}\delta_S\Big(f\big({n_m,n_s-1},(z_1,\dots,z_{i-1},z_{i+1},\dots,z_{n_m+n_s})\big)-f(n_m,n_s,(z_p)_{p=1}^{n_m+n_s})\Big)\\
            &+\sum_{j=n_m+1}^{n_m+n_s}\sum_{k=1}^\infty\nu(\{k\})\Big(f\big(n_m+k,n_s,(z_1,\dots,z_{n_m},\underbrace{z_{j},\dots,z_j}_{\text{$k$ times}} , z_{n_m+1},\dots,z_{n_m+n_s})\big)\\
            &\hspace{20em}-f(n_m,n_s,(z_p)_{p=1}^{n_m+n_s})\Big)\\
            &+\sum_{i=1}^{n_m}G_{z_i} f(n_m,n_s,(z_p)_{p=1}^{n_m+n_s}),
        \end{align*}
        where $\mathfrak{D}\subset \DD(\AA)$ consists of all $f\in C_0(\wt \XX,\wt d_{\wt \XX})$ such that for all $(n_m,n_s)\in \N_0^2$ it holds that $z_i\mapsto f(n_m,n_s,(z_p)_{p=1}^{n_m+n_s})\in \DD(G)$ and $G_{z_i}f(n_m,n_s,(z_p)_{p=1}^{n_m+n_s})\in C_0(\wt X,\wt d_{\wt X})$.
    \end{theorem}
    \begin{proof}
        By \cite[Theorem 1.33]{LevyMatters3}, it is enough to study the pointwise limits in \eqref{eq:generator} and see for which functions $f$ the limit exists and gives as a result a $C_0(\XX)$ function.

        To this end, take $f\in C_0(\XX)$ and $x=(n_m,n_s,(z_p)_{p=1}^{n_m+n_s})\in \XX$. We study the expression $\frac{P_tf(x)-f(x)}{t}=\frac{\E^xf(X(t))-f(x)}{t}$. 
        
        Assume that in the interval $[0,t]$ only the $i$-th, for $1\le i\le n_m$, moving particle dies, while the rest of them do not change in any way (except for the positions of the moving particles), and call this event $D_i$. Since the changes are independent of positions and other particles, we have
        \begin{align}
            &\frac{\E^x[f(X(t));D_i]}{t}\nonumber\\
            &=\E^x[f(n_m-1,n_s,(Z_1(t),\dots,Z_{i-1}(t),Z_{i+1}(t),\dots,, Z_{n_m}(t), z_{n_m},\dots, z_{n_m+n_s}));D_i]/t\nonumber\\
            &=\frac{(1-e^{-\delta_Mt})}{t} e^{-(n_m-1)(\delta_M+\mu_S) t}e^{-(n_s)(\delta_S+\mu_B) t}\nonumber\\&
            \hspace{2em}\times \E^x[f(n_m-1,n_s,(Z_1(t),\dots,Z_{i-1}(t),Z_{i+1}(t),\dots,, Z_{n_m}(t), z_{n_m},\dots, z_{n_m+n_s}) )]\nonumber\\
            &\to \delta_M f\big(n_m-1,n_s,(z_1,\dots,z_{i-1},z_{i+1},\dots,z_{n_m+n_s})\big),\quad \text{as $t\searrow 0$,}\label{death-limit}
        \end{align}
        and note that there are no other regularity assumptions on $f$.

        Similarly, if we assume that only the settled particle $j$ dies in $[0,t]$, while all the other do not change statuses, the event which we call $D_j$ with $n_m<j\le n_m+n_s$, then we get
        \begin{align}
            &\frac{\E^x[f(X(t));D_j]}{t}\nonumber\\
            &\to \delta_S f\big(n_m,n_s-1,(z_1,\dots,z_{j-1},z_{j+1},\dots,z_{n_m+n_s})\big),\quad \text{as $t\searrow 0$.}\label{death-limit-settle}
        \end{align}
        
        Further, assume now that only the moving particle $i$ settles in $[0,t]$, gives birth to $\xi_0$ offspring and does not produce additional new offspring, while all the other do not change statuses. The probability that the particle $i$ settles in $[0,t]$ and does not have additional offsprings besides $\xi_0$ equals {to
        \begin{align}
            &\P^x(\EXP_p(\mu_S)\le t, \EXP_p(\mu_B)\ge t-\EXP_p(\mu_S))\\
            &=\int_0^t \left(\int_{t-s}^\infty\mu_Be^{-\mu_B h}dh\right)\mu_S e^{-\mu_S s}ds=e^{-\mu_B t}\frac{\mu_S}{\mu_S-\mu_B}\big(1-e^{-(\mu_s-\mu_b)t}\big),   
        \end{align}
        for any $x$.} Therefore, it is easy to see that we get as $t\searrow 0$ that
        \begin{align}
            &\frac{\E^x[f(X(t));S_i]}{t}\nonumber\\
            &\to \sum_{k=0}^\infty q_k\mu_S f\big({n_m-1+k},n_s+1,(z_1,\dots,z_{i-1},z_{i+1},\dots,z_{n_m},\underbrace{z_i,\dots,z_i}_{\text{$k$  times}},z_{n_m+1},\dots, z_{n_m+n_s},z_{i})\big),\label{limit-settle}
        \end{align}
        where we recall the notation $\P(\xi_0=k)=q_k$.

        Further, if we look at the event that only the settled particle $j$, where $n_m<j\le n_m+n_s$, produced the offspring(s) in $[0,t]$, while all the other did not make a change in the statuses, the event denoted by $B_j$, we have
        \begin{align}
            &\frac{\E^x[f(X(t));B_j]}{t}\nonumber\\
            &=\sum_{k=1}^{\infty}\E^x[f(n_m+k,n_s,(Z_1(t),\dots,Z_{n_m}(t),\underbrace{\wt Z_{j}(t),\dots,\wt Z_j(t)}_{\text{$k$ times}} , z_{n_m+1},\dots,z_{n_m+n_s}));B_j]/t\nonumber\\
            &=\sum_{k=1}^{\infty}e^{-n_m(\delta_M+\mu_S)t}e^{-(n_s-1)(\delta_S+\mu_B)t}\frac{(1-e^{-\mu_B t})}{t}\frac{\nu(\{k\})}{\mu_B} \times\notag\\
            &\hspace{4em}\times \E^x[f(n_m+k,n_s,(Z_1(t),\dots,Z_{n_m}(t),\underbrace{\wt Z_{j}(t),\dots,\wt Z_j(t)}_{\text{$k$ times}} , z_{n_m+1},\dots,z_{n_m+n_s})]\nonumber\\
            &\to \sum_{k=1}^\infty\nu(\{k\})f\big(n_m+k,n_s,(z_1,\dots,z_{n_m},\underbrace{z_{j},\dots,z_j}_{\text{$k$ times}} , z_{n_m+1},\dots,z_{n_m+n_s})\big),\quad \text{as $t\searrow 0$,}\label{birth-gen}
        \end{align}

        If in the interval $[0,t]$ at least two of the former events (settlement, birth, death) happen, for the corresponding ratio we have $$\frac{\E^x[f(X(t));\text{at least two events happen}]}{t}\to 0.$$
        The argument is as follows: if at least two events happen, then by repeating the  calculations as in \eqref{death-limit} and \eqref{birth-gen} (in particular, the the third lines therein) will contain at least two exponentials of the form $1-e^{-\delta_M t}$ or $1-e^{-\delta_S t}$ or $1-e^{-\mu_B t}$ in the numerator, while $t$ will stay in the denominator, so the whole expression will be bounded from above by e.g. $\|f\|\frac{1-e^{-\delta_M t}}{t}(1-e^{-\delta_S t})\to0$.

        On the other hand, if we assume that there are no events in $[0,t]$ but only movements, the event which we denote by $N_0$, we have
        \begin{align}
            &\E^x[f(X(t));N_0]\nonumber\\
            &=e^{-n_m(\delta_M+\mu_S)t}e^{-n_s(\delta_S+\mu_B)t}\E^x[f(n_m,n_s,( Z_1(t),\dots,Z_{n_m}(t), z_{n_m+1},\dots,z_{n_m+n_s}))].
        \end{align}

        Thus, for $f\in C_0(\wt \XX,\wt d_{\wt \XX})$ such that for all $(n_m,n_s)\in \N_0^2$ it holds that $z_i\mapsto f(n_m,n_s,(z_p)_{p=1}^{n_m+n_s})\in \DD(G)$ and $G_{z_i}f(n_m,n_s,(z_p)_{p=1}^{n_m+n_s})\in C_0(\wt X,\wt d_{\wt X})$, we have by Lemma \ref{ap:l:generator} 
        \begin{align}
        \begin{split}\label{gengen}
            &\frac{\E^x[f(X(t));N_0]-e^{-n_m(\delta_M+\mu_S)t}e^{-n_s(\delta_S+\mu_B)t}f(x)}{t}\\&\qquad\qquad\to \sum_{i=1}^{n_m}G_{z_i} f(n_m,n_s,(z_p)_{p=1}^{n_m+n_s}),\quad \text{as $t\searrow0$.}
        \end{split}
        \end{align}

        Therefore, by noticing that 
        $$\lim_{t\searrow0}\frac{e^{-n_m(\delta_M+\mu_S)t}e^{-n_s(\delta_S+\mu_B)t}f(x)-f(x)}{t}=f(x)\big(n_m(\delta_M+\mu_S)+n_s(\delta_S+\mu_B)\big),$$
        and combining with what was calculated in \eqref{death-limit}, \eqref{death-limit-settle}, \eqref{limit-settle}, \eqref{birth-gen}, and \eqref{gengen}, we obtain that 
        \begin{align*}
            &\frac{P_tf(x)-f(x)}{t}=\frac{\E^xf(X(t))-f(x)}{t}\\
            &\to \sum_{i=1}^{n_m}\delta_M \Big(f\big(n_m-1,n_s,(z_1,\dots,z_{i-1},z_{i+1},\dots,z_{n_m+n_s})\big)-f(n_m,n_s,(z_p)_{p=1}^{n_m+n_s})\Big)\\
            &+\sum_{i=1}^{n_m}\sum_{k=0}^\infty q_k\mu_S\Big(f\big({n_m-1+k},n_s+1,(z_1,\dots,z_{i-1},z_{i+1},\dots,z_{n_m},\underbrace{z_i,\dots,z_i}_{\text{$k$  times}},z_{n_m+1},\dots, z_{n_m+n_s},z_{i})\big)\notag\\
            &\hspace{20em}-f(n_m,n_s,(z_p)_{p=1}^{n_m+n_s})\Big)\\
            &+\sum_{j=n_m+1}^{n_m+n_s}\delta_S\Big(f\big(n_m,{n_s-1},(z_1,\dots,z_{i-1},z_{i+1},\dots,z_{n_m+n_s})\big)-f(n_m,n_s,(z_p)_{p=1}^{n_m+n_s})\Big)\\
            &+\sum_{j=n_m+1}^{n_m+n_s}\sum_{k=1}^\infty\nu(\{k\})\Big(f\big(n_m+k,n_s,(z_1,\dots,z_{n_m},\underbrace{z_{j},\dots,z_j}_{\text{$k$ times}} , z_{n_m+1},\dots,z_{n_m+n_s})\big)\\
            &\hspace{20em}-f(n_m,n_s,(z_p)_{p=1}^{n_m+n_s})\Big)\\
            &+\sum_{i=1}^{n_m}G_{z_i} f(n_m,n_s,(z_p)_{p=1}^{n_m+n_s}), \quad \text{as $t\searrow 0$,}
        \end{align*}
        for all $f\in C_0(\wt \XX,\wt d_{\wt \XX})$ such that for all $(n_m,n_s)\in \N_0^2$ it holds that $z_i\mapsto f(n_m,n_s,(z_p)_{p=1}^{n_m+n_s})\in \DD(G)$ and $G_{z_i}f(n_m,n_s,(z_p)_{p=1}^{n_m+n_s})\in C_0(\wt X,\wt d_{\wt X}).$
    \end{proof}
	
	\paragraph*{Acknowledgments}
	The authors acknowledge financial support under the National Recovery and Resilience Plan (NRRP), Mission 4, Component 2, Investment 1.1, Call for tender No. 104 published on 2.2.2022 by the Italian Ministry of University and Research (MUR), funded by the European Union – NextGenerationEU– Project Title “Non–Markovian Dynamics and Non-local Equations” – 202277N5H9 - CUP: D53D23005670006 - Grant Assignment Decree No. 973 adopted on June 30, 2023, by the Italian Ministry of University and Research (MUR).

    Bruno Toaldo is partially supported by Gruppo Nazionale per l’Analisi Matematica, la Probabilità e le loro Applicazioni (GNAMPA-INdAM).
	
	\appendix
	
	\section{Tensor semigroups}
    Here, we consider a locally compact separable metric space $E$, and a strongly continuous semigroup of operators on $C_0(E)$.
		\begin{definition}[Tensor product of a Feller semigroup] \label{def:TensorFeller}
			Let $n\in\N$ and $(T_t)_t$ be a semigroup on $C_0(E)$. Then, we define the $n$-fold tensor product of $(T_t)_t$ at times $t_1, \dots, t_n\in[0,\infty)$ as $T^{\otimes n}_{t_1, \dots, t_n}$ with
			\begin{equation*}
				\left(T^{\otimes n}_{t_1, \dots, t_n} v\right)(x_1, \dots, x_n) = \int_E \cdots \int_E v(y_1, \dots, y_n) p_{t_1}(x_1, dy_1) \cdots p_{t_n}(x_n, dy_n), 
			\end{equation*}
			for all $(x_1,\dots,x_n)\in E^n$ and for all $v \in C_0(E^n)$.
		\end{definition}
		
		\begin{lemma}[Tensor products of a Feller semigroups are Feller]\label{lem:TensorProductFellerSemigroup}
			Let $(T_t)_t$ be a Feller semigroup on $C_0(E)$ with the transition kernel $p_t(x,dy)$, $t\ge0$, $x\in E$. Then, the $n$-fold tensor product of $(T_t)_t$ at times $t_1, \dots, t_n$, denoted by $T^{\otimes n}_{t_1, \dots, t_n}$ as in Definition~\ref{def:TensorFeller}, preserves the Feller properties in the following sense: $T^{\otimes n}_{t_1, \dots, t_n}$ is non-negative contraction; it satisfies $T^{\otimes n}_{t_1, \dots, t_n}(T^{\otimes n}_{t_1, \dots, t_n})=T^{\otimes n}_{t_1+s_1, \dots, t_n+s_n}$ and $T^{\otimes n}_{t_1, \dots, t_n}(C_0(E^n))\subseteq C_0(E^n)$; and it is strongly continuous, i.e. for every $v\in C_0(E^n)$ and every $\epsilon > 0$, there exists $\delta > 0$ such that if $\max_{i=1, \dots, n} t_i < \delta$, we have
				\begin{equation*}
					\left \lVert T_{t_1, \dots, t_n}^{\otimes n} v - v \right\rVert_{\infty} < \epsilon.
				\end{equation*}
		\end{lemma}
		
		\begin{proof}
			The proof is almost straightforward from the definitions and is left to the reader. The nontrivial part is the preservation of $C_0(E^n)$ and the strong continuity which can be proved by Stone-Weierstrass' theorem of approximation  for locally
			compact spaces, see e.g. \cite[Theorem A.0.5]{applebaum-semigroups}. 
		\end{proof}

        \section{Auxiliary results}

        \begin{lemma}\label{ap:l:feller-escaping}
            Let $(Z(t))_{t}$ be a Feller process on locally compact separable metric space $E$. Then for $\varepsilon>0$, $t>0$ and a compact $K\subset E$, there exists an another compact $\wt K\subset E$ such that
            \begin{align}
                \P^x(Z(s) \in K)<\varepsilon,\quad x\notin \wt K, \, s\le t.   
            \end{align}
        \end{lemma}
        \begin{proof}
            Denote by $(T_t)_t$ the strongly continuous semigroup on $C_0(E)$ associated with $(Z(t))_t$. Fix $\varepsilon>0$, $t>0$ and a compact $K\subset E$. 
            
            By Urysohn's lemma, there exists $f\in C_c(E)$ such that $0\le f\le 1$, $f\equiv 1$ on some neighborhood of $K$.
            Since $T_tf\to f$ strongly as $t\to 0$, there exists $t_\varepsilon>0$ such that for $K_1\coloneqq \supp f$ it holds 
            \begin{align}\label{eq:1513a}
                T_s \1_K(x) \le T_s f(x) \le \frac\varepsilon2, \quad x\notin K_1, \, s\in [0,t_\varepsilon],
            \end{align}
            where the first inequality holds for all $x\in E$.
            
            If $t_\varepsilon\ge t$, the proof is finished. If not, for $s\in [0,t_\varepsilon]$ we have
            \begin{align}\label{eq:1513b}
                T_{t_\varepsilon+s} \1_K(x)\le T_{t_\varepsilon+s} f(x) \le T_{t_\varepsilon}\big(T_s f \1_{K_1} + T_s f \1_{K_1^c}\big)\le T_{t_\varepsilon}\1_{K_1}(x)+\frac\varepsilon2,
            \end{align}
            where we used $\|T_s f\|\le \|f\|=1$ for the first term in the last inequality and \eqref{eq:1513a} for the second term. By the Feller property (i.e. by using again Urysohn's lemma), there exists a compact $K_2\supset K_1$ such that $T_{t_\varepsilon}\1_{K_1}(x)\le \frac{\varepsilon}{4}$, for all $x\notin K_2$. Therefore, by using this in \eqref{eq:1513b}, we get that for all $x\notin K_2$ and all $s\in [0,2t_\varepsilon]$ we have
            \begin{align}\label{eq:1514a}
                T_s \1_K(x)  \le T_s f\le \frac34\varepsilon, \quad x\notin K_2, \, s\in [0,2t_\varepsilon].
            \end{align}

            We can now bootstrap this argument. For $s\in [0,2t_\varepsilon]$ we have
            \begin{align}\label{eq:1514b}
                T_{t_\varepsilon+s} \1_K(x)\le T_{t_\varepsilon+s} f(x)\le T_{t_\varepsilon}\big(T_s f \1_{K_2} + T_s f \1_{K_2^c}\big)\le T_{t_\varepsilon}\1_{K_2}(x)+\frac34\varepsilon.
            \end{align}
            Further, there exists a compact $K_3\supset K_2$ such that $T_{t_\varepsilon}\1_{K_2}(x)\le \frac{\varepsilon}{2^3}$, for all $x\notin K_3$. Therefore, by using this in \eqref{eq:1514b}, we get that it holds
            \begin{align}\label{eq:1514a1}
                T_s \1_K(x)  \le T_sf(x)\le \frac78\varepsilon, \quad x\notin K_3, \, s\in [0,3t_\varepsilon].
            \end{align}

            We stop this argument after $k\in \N$ {steps, i.e., until $k\,t_\varepsilon>t$,} since at that point we constructed a compact $K_k\subset E$ such that it holds
            \begin{align}\label{eq:1514a2}
                T_s \1_K(x)  \le \frac{2^k-1}{2^k}\varepsilon, \quad x\notin K_k, \, s\in [0,kt_\varepsilon]
            \end{align}
            {which is the claimed statement for $\wt K = K_k$.}
        \end{proof}

        \begin{lemma}\label{ap:l:generator}
        
            Let $X=(X_t)_t$ and $Y=(Y_t)_t$ be independent Feller processes on locally compact and separable metric space $E$, with the infinitesimal generators $G^{X}$ and $G^{Y}$, respectively. Then the process $Z=({Z}_t)_t=(X_t,Y_t)_{t}$ is a Feller process  on $E\times E$ with the generator $(G,\DD(G))$ such that
            \begin{align}
                Gf(x,y)=G^X\big(f(\cdot,y)\big)(x)+
                G^Y\big(f(x,\cdot)\big)(y),
            \end{align}
            for all $f\in C_0(E\times E)$ such that $f(\cdot,y)\in \DD(G^X)$ and $f(x,\cdot)\in \DD(G^Y)$ with $G^X\big(f(\cdot,y)\big)(x)+
                G^Y\big(f(x,\cdot)\big)(y)\in C_0(E\times E)$.
        \end{lemma}
        \begin{proof}
            \begin{align*}
                \frac{\E^{(x,y)}f(X_t,Y_t)-f(x,y)}{t}&=\frac{\E^{(x,y)}f(X_t,y)-f(x,y)}{t}\\&\hspace{6em}+\int_E\frac{\E^{(y)}f(w,Y_t)-f(w,y)}{t}\P^{x}(X_t\in dw)\\&\to G^X\big(f(\cdot,y)\big)(x)+G^Y\big(f(x,\cdot)\big)(y).
            \end{align*}
            The convergence of the second term if due to generalized dominated convergence theorem: if $f_n\to f$ uniformly on compact subsets of a Polish space $\mathbb S$, $f_n$ are uniformly bounded, and $\mu_n\to \mu$ weakly where $\mu_n$ are finite measures, then $\int_E f_n d\mu_n\to \int_E fd\mu$, easily proved by using Prokhorov's theorem.

            Hence, the domain of the generator certainly contains $f\in C_0(E\times E)$ such that $f(\cdot,y)\in \DD(G^X)$ and $f(x,\cdot)\in \DD(G^Y)$ with $G^X\big(f(\cdot,y)\big)(x)+
                G^Y\big(f(x,\cdot)\big)(y)\in C_0(E\times E)$.
        \end{proof}

	\bibliographystyle{ieeetr}
	\bibliography{literature_paper}
	

	\bigskip
	
	\noindent{\bf Ivan Bio\v{c}i\'c}
	
	\noindent Department of Mathematics, Faculty of Science, University of Zagreb, Zagreb, Croatia,
	
	\noindent Email: \texttt{ibiocic@math.hr}
    
        \noindent Department of Mathematics ``Giuseppe Peano", University of Turin, Turin, Italy,
	
	\noindent Email: \texttt{ivan.biocic@unito.it}
	
	\bigskip
	
	\noindent{\bf Bruno Toaldo}
	
	\noindent Department of Mathematics ``Giuseppe Peano", University of Turin, Turin, Italy,
	
	\noindent Email: \texttt{bruno.toaldo@unito.it}

    \bigskip
	
	\noindent{\bf Lena Zuspann}


    \noindent Mathematical Institute, University of Oxford, Oxford, UK,

	\noindent Email: \texttt{lena.zuspann@maths.ox.ac.uk}

    \noindent Department of Mathematics ``Giuseppe Peano", University of Turin, Turin, Italy.
\end{document}